\documentclass{article}

\usepackage{amssymb}
\usepackage{amsmath}
\usepackage{amsthm}
\usepackage{mathrsfs}
\usepackage{hyperref}
\usepackage{amsrefs}
\usepackage{comment}

\usepackage[margin=1in]{geometry}
\usepackage{setspace}
\newtheorem{theorem}{Theorem}[section]
\newtheorem{corollary}{Corollary}[theorem]
\newtheorem{lemma}[theorem]{Lemma}
\theoremstyle{definition}
\newtheorem{definition}[theorem]{Definition}
\newtheorem{assumption}[theorem]{Assumption}
\newtheorem{example}{Example}
\newtheorem*{question}{Question}

\newcommand{\p}{\mathfrak{p}}
\newcommand{\m}{\mathfrak{m}}
\newcommand{\fq}{\mathfrak{q}}
\newcommand{\fa}{\mathfrak{a}}
\newcommand{\fb}{\mathfrak{b}}
\newcommand{\spec}{\operatorname{Spec}}
\newcommand{\height}{\operatorname{ht}}
\newcommand{\ass}{\operatorname{Ass}}
\newcommand{\image}{\operatorname{Image}}
\newcommand{\depth }{\operatorname{depth}}
\def\cal{\mathcal}

\begin{document}
\title{Controlling the Dimensions of Formal Fibers of a Unique Factorization Domain at the Height One Prime Ideals \thanks{The authors thank the National Science Foundation for their support of this research via grant DMS-1347804.} \thanks{Sarah M. Fleming and Nina Pande thank the Clare Boothe Luce Foundation for support of their research.}}
\author{Sarah M. Fleming, Lena Ji, S. Loepp, Peter M. McDonald, Nina Pande, and David Schwein}

\maketitle

\begin{abstract}
Let $T$ be a complete local (Noetherian) equidimensional ring with maximal ideal $\m$ such that the Krull dimension of $T$ is at least two and the depth of $T$ is at least two.  Suppose that no integer of $T$ is a zerodivisor and that $|T|=|T/\m|$.  Let $d$ and $t$ be integers such that $1\leq d \leq  \dim T-1$, $0 \leq t \leq \dim T - 1$, and $d - 1 \leq t$.  Assume that, for every $\p\in\ass T$, $\height\p\leq d-1$ and that if $z$ is a regular element of $T$ and $Q \in \ass(T/zT)$, then $\height Q \leq d$.  
We construct a local unique factorization domain $A$ such that the completion of $A$ is $T$ and such that the dimension of the formal fiber ring at every height one prime ideal of $A$ is $d - 1$ and the dimension of the formal fiber ring of $A$ at $(0)$ is $t$.
\end{abstract}

\section{Introduction}

Let $A$ be a local (Noetherian) ring and let $T$ be the completion of $A$ with respect to its maximal ideal.  One fruitful way to study the relationship between $A$ and $T$ is to examine the relationship between $\spec A$ and $\spec T$, which we can do by looking at the formal fibers of $A$ at its prime ideals.  We define the formal fiber ring of $A$ at $\p\in\spec A$ as $T\otimes_A\kappa(\p)$, and we define the formal fiber of $A$ at $\p$ as $\spec(T\otimes_A\kappa(\p))$ where $\kappa(\p)$ is the residue field $A_{\p}/\p A_{\p}$. There is a one-to-one correspondence between elements of the formal fiber of $A$ at $\p$ and the elements of the inverse image of $\p$ under the mapping $\varphi : \spec T\rightarrow\spec A$ defined by $\varphi(\fq)=\fq\cap A$.  Because of this correspondence, the formal fibers of $A$ encode information about the relationship between the prime ideals of $T$ and the prime ideals of $A$.  Note that, since $T$ is a faithfully flat extension of $A$, the map $\varphi$ is surjective.  Using Matsumura's notation in~\cite{matsumura88}, we denote the Krull dimension of the formal fiber ring of $A$ at the prime ideal $\p$ by $\alpha(A,\p)$.  The supremum of $\alpha(A,\p)$ over all prime ideals $\p$ of $A$ is denoted by $\alpha(A)$.  When $A$ is an integral domain, we refer to the formal fiber ring of $A$ at the prime ideal $(0)$ as the generic formal fiber ring, and the formal fiber of $A$ at $(0)$ as the generic formal fiber.

Matsumura proved in~\cite{matsumura88} that, given two elements $\p,\p'\in\spec A$, if $\p\subseteq\p'$, then $\alpha(A,\p)\geq\alpha(A,\p')$.  Additionally, he proved that, given a ring of positive dimension $n$, the maximum possible value of $\alpha(A)$ is $n-1$.  Furthermore, if $A$ is a ring of essentially finite type over a field, Matsumura showed that $\alpha(A) = n - 1$ and that $\alpha(A,\p)=n-1-\height\p$.  In other words, for rings of essentially finite type over a field, the dimensions of the formal fiber rings are completely understood.  However, for a general ring, not much is known about how the inequality $\alpha(A,\p)\geq\alpha(A,\p')$ behaves when varying the heights of $\p$ and $\p'$.  We are interested specifically in what happens when $A$ is an integral domain, $\p=(0)$, and $\p'$ is a height one prime ideal of $A$.  In particular, Heinzer, Rotthaus, and Sally informally posed the following question:

\begin{question}
Let A be an excellent local ring with $\alpha(A)>0$, and let $\Delta$ be the set $\{\p\in\spec A\mid\height\p=1\text{ and }\alpha(A,\p) = \alpha(A)\}$.
Is $\Delta$ a finite set?
\end{question}

\noindent We first note that if $A$ is a Noetherian integral domain, then $\alpha{(A)} = \alpha(A,(0))$.  Given Matsumura's results, one might expect that, for an integral domain $A$, the dimension of the formal fiber ring of $A$ at ``most" height one prime ideals is strictly less than the dimension of the formal fiber ring of $A$ at the zero ideal.  In other words, one might expect that the set $\Delta$ is ``small," so the above question is a very natural one to ask.

In~\cite{boocher10}, Adam Boocher, Michael Daub, and S. Loepp showed that the set $\Delta$ need not be finite by constructing an excellent local unique factorization domain $A$ for which $\Delta$ is countably infinite.  However, the UFD that they construct has uncountably many height one prime ideals, so the set $\Delta$ for their ring is ``small" in the sense that it does not contain ``most" of the height one prime ideals of $A$.  Based on this result, one might think it unlikely that there exists a ring such that $\Delta$ is an uncountable set and even less likely that a ring exists such that $\Delta$ contains {\em all} of its height one prime ideals.  In this paper, we construct a non-excellent unique factorization domain $A$ for which this property holds; that is, every height one prime ideal of $A$ is in $\Delta$.  In fact, we generalize the result so that if $d$ and $t$ are integers such that $1 \leq d \leq \dim T - 1$, $0 \leq t \leq \dim T - 1$, and $d - 1 \leq t$, then our $A$ satisfies the property that every height one prime ideal $\p$ of $A$ has a formal fiber ring of dimension $d - 1$ and the generic formal fiber ring of $A$ has dimension $t$.  In other words, we show that the relationship between the dimensions of the formal fiber rings of a ring at its height one prime ideals and the dimension of the generic formal fiber ring of the ring is restricted only by the inequality given by Matsumura.  That is, if $\p$ is a height one prime ideal of a ring $A$, then $\alpha (A,\p) \leq \alpha (A,(0))$, but
for any nonnegative integer less than or equal to the dimension of the generic formal fiber ring,
there exists a ring such that the dimension of the formal fiber rings at each height one prime ideals is equal to this integer.

Our main result is given by Theorem~\ref{bigthm}.  Let $T$ be a complete local equidimensional ring with maximal ideal $\m$, and suppose that the Krull dimension of $T$ is at least two.  Also suppose that $|T|=|T/\m|$, no integer is a zerodivisor in $T$, and $T$ has depth greater than one. Let $d$ and $t$ be integers with $1\leq d\leq\dim T-1$, $0 \leq t \leq \dim T - 1$, and $d - 1 \leq t$.  Suppose that $\height\p\leq d-1$ for every $\p\in\ass T$ and that if $z$ is a regular element of $T$ and $Q \in \ass(T/zT)$, then $\height Q \leq d$.  We show that there exists a local unique factorization domain $A$ whose completion is $T$ such that $\alpha(A,(0))=t$ and, for every $\p\in\spec A$ with $\height\p=1$, $\alpha(A,\p)= d - 1$.  
That is to say, with relatively weak conditions on $T$, we can find a local UFD $A$ where the dimensions of the formal fiber rings of $A$ at all height one prime ideals are equal.  Moreover, we can control the dimension of the generic formal fiber ring of $A$, and the dimension of the formal fiber rings of $A$ at its height one prime ideals can be any nonnegative integer less than or equal to the dimension of the generic formal fiber ring of $A$.
In particular, if we set $d - 1 = t$, then we can find a UFD $A$ such that $\Delta$ is the set of {\em all} height one prime ideals of $A$, which, in our case, will be an uncountable set.  Furthermore, in this case, we have the ability to set the dimension of the formal fiber rings at our height one prime ideals and the generic formal fiber ring to be any integer value between $0$ and $\dim T-2$.

In ~\cite{heitmann93}, Heitmann found necessary and sufficient conditions for a complete local ring to be the completion of a unique factorization domain.  For our ring $A$ to exist, then, our ring $T$ must satisfy those conditions.  In particular, this means that $T$ must have depth greater than one and satisfy the condition that no integer of $T$ is a zerodivisor.  We impose the additional condition that $\dim T\geq2$ to avoid any trivial examples. The remaining conditions we put on $T$ are relatively weak.  We base our construction on the one employed by Heitmann in~\cite{heitmann93}, modifying and adding lemmas so that we can control the dimensions of the formal fiber rings of $A$ at $(0)$ and at the height one prime ideals.  
We employ Lemma~\ref{comp}, which gives necessary and sufficient conditions for a quasi-local subring of a complete local ring $T$ to be Noetherian and have completion $T$.  In particular, if $T$ is a complete local ring with maximal ideal $\m$ and $A$ is a quasi-local subring of $T$ with maximal ideal $\m \cap A$, then $A$ is Noetherian and has completion isomorphic to $T$ if the map $A\rightarrow T\slash\m^2$ is onto and every finitely generated ideal $\fa$ of $A$ is closed; that is, $\fa T\cap A=\fa$.  To ensure that our subring $A$ satisfies these conditions and has the desired formal fibers, we construct a strictly ascending chain of subrings of $T$ starting with a localization of the prime subring of $T$.  We then adjoin elements in of $T$ in succession to make the above map onto, to close up finitely generated ideals, and to manipulate the dimensions of the formal fiber rings.  We argue that the union of the subrings in this chain is our desired local UFD $A$.

In this paper, all rings are commutative with unity.  We define a quasi-local ring as any ring with exactly one maximal ideal and a local ring as a quasi-local Noetherian ring.  We use $(T,\m)$ to refer to a local ring $T$ with maximal ideal $\m$ and $\spec_k T$ to refer to the set of prime ideals of $T$ of height $k$.  $\widehat{T}$ will denote the completion of $T$ in the $\m$-adic topology.

\section{Preliminaries}

Before we begin our construction, we present, without proof, several lemmas which we use throughout the construction.  The following lemma will be used to show that the completion of $A$ is indeed $T$.

\begin{lemma}[{\cite[Proposition 1]{heitmann94}}]
\label{comp}
If $(A,\m\cap A)$ is a quasi-local subring of a complete local ring $(T,\m)$, $A\rightarrow T\slash\m^2$ is onto and $\fa T\cap A=\fa$ for every finitely generated ideal $\fa$ of $A$, then $A$ is Noetherian and the natural homomorphism $\widehat{A}\rightarrow T$ is an isomorphism.
\end{lemma}

We use the following lemma to find prime ideals of $T$ of height $d$ to include in the formal fibers at the height one prime ideals of $A$.

\begin{lemma}[{\cite[Lemma 2.3]{charters04}}]
\label{modcard}
Let $(T,\m)$ be a complete local ring of dimension at least one.  Let $\p$ be a nonmaximal prime ideal of $T$.  Then $|T/\p|=|T|\geq 2^{\aleph_0}$.
\end{lemma}

Throughout our construction, we adjoin elements of $T$ to our intermediate subrings.  The following two lemmas ensure that we can adjoin the necessary elements of $T$ while also preserving the desired properties of our subrings.

\begin{lemma}[{\cite[Lemma 2]{heitmann93}}]
\label{countavoidance}
Let $(T,\m)$ be a complete local ring and let $D\subset T$.  Suppose $C\subset\spec T$ such that $\m\not\in C$, and suppose $\fa$ is an ideal of $T$ such that $\fa\not\subseteq\p$ for all $\p\in C$.  If $C$ and $D$ are countable, then
\[
\fa\nsubseteq\bigcup\{t+\p\mid t\in D, \p\in C\}.
\]
\end{lemma}

\begin{lemma}[{\cite[Lemma 3]{heitmann93}}]
\label{avoidance}
Let $(T,\m)$ be a local ring and let $D\subset T$.  Suppose $C\subset\spec T$ and $\fa$ is an ideal of $T$ such that $\fa\not\subseteq\p$ for all $\p\in C$.  If $|C\times D|<|T/\m|$, then 
\[
\fa\nsubseteq\bigcup\{t+\p\mid t\in D, \p\in C\}.
\]
\end{lemma}

\section{The Construction}

Here we describe the construction of our local UFD $A$.  We begin with the definition of a $Z_d$-subring of $T$.  If we can maintain $Z_d$-subrings throughout the construction, then the formal fiber rings of $A$ will be of the desired dimension.

\begin{definition}\label{zsub} 
Let $(T,\m)$ be a complete local ring, let $d$ be an integer such that $1\leq d\leq \dim T-1$, and let $(R,\m\cap R)$ be a quasi-local unique factorization domain contained in $T$ satisfying the following:
\begin{description}
\item{(i)} $|R|\leq\sup(\aleph_0,|T/\m|)$ with equality only if $T/\m$ is countable,
\item{(ii)} $\fq\cap R=(0)$ for all $\fq\in\ass T$,
\item{(iii)} if $t\in T$ is regular and $\fq\in\ass(T/tT)$, then $\height(\fq\cap R)\leq 1$, and
\item{(iv)} there exists a set $\cal{Q}_R \subset \spec_d T$ such that the map $f: \cal{Q}_R \longrightarrow \spec_1 R$ given by $f(\fq) = \fq \cap R$ is a bijection.
\end{description}
\noindent Then $R$ is called a \emph{$Z_d$-subring of $T$ with distinguished set $\cal{Q}_R$}.  
 Note that if $(R, \m\cap R)$ satisfies conditions (i)-(iii) but not necessarily (iv), $R$ is an \emph{$N$-subring of $T$}, as defined in~\cite{heitmann93}.
\end{definition}

\begin{definition} If $R\subseteq S$ are $Z_d$-subrings of $T$ with distinguished sets $\cal{Q}_R$ and $\cal{Q}_S$, we say that $S$ is an \emph{$A_+$-extension of $R$} if:
\begin{description}
\item{(i)} prime elements of $R$ are prime in $S$,
\item{(ii)} $|S|\leq\sup(\aleph_0,|R|)$, and 
\item{(iii)} $\cal{Q}_R\subseteq\cal{Q}_S$.
\end{description}
\noindent If $R\subseteq S$ are $N$-subrings of $T$ satisfying conditions (i) and (ii) but not necessarily (iii), then $S$ is an \emph{$A$-extension of $R$}, as defined in~\cite{heitmann93}.
\end{definition}

In order to simplify the statements of many of our subsequent lemmas, we will refer to the following assumption.

\begin{assumption} \label{assume}
$(T,\m)$ is a complete local equidimensional ring with $2\leq\dim T$, $d$ is an integer such that
$1\leq d\leq\dim T-1$, and $(R,\m\cap R)$ is a $Z_d$-subring of~$T$ with distinguished set $\cal{Q}_R$.
\end{assumption}

The next lemma is used to adjoin a coset representative of an element of $T/\m^2$ to our subring $R$.  Applying this lemma infinitely often will allow us to ensure that the map $A\rightarrow T/\m^2$ is onto so that we can employ Lemma \ref{comp}.

\begin{lemma}\label{cosets}
Under Assumption~\ref{assume}, let $t\in T$ and suppose $\depth T>1$.  Let $P$  be a nonmaximal prime ideal of $T$ such that $P \cap R = (0)$.  Then there exists an $A_+$-extension $S$ of $R$ such that $t+\m^2\in\image(S\rightarrow T/\m^2)$ and $P \cap S = (0)$.
\end{lemma}
\begin{proof}
Note that this lemma is similar to Lemma 5 of~\cite{heitmann93}. Let $C=\{\p\in\spec T\mid\p\in\ass(T/rT)\text{ with }0\neq r\in R\}\cup\ass T\cup\mathcal{Q}_R \cup \{P\}$.  Since $\depth T > 1$, we have that $\m \not\in C$, so $\m^2\not\subset \p$ for every $\p\in C$.  For $\p \in C$, let $D_{(\p)}$ be a set of coset representatives of the cosets $u + \p \in T/\p$ that make $u + t + \p$ algebraic over $R/\p \cap R$.  Define $D = \bigcup_{\p \in C} D_{(\p)}$.  Using Lemma~\ref{countavoidance} if $R$ is countable and Lemma~\ref{avoidance} otherwise, we choose $x\in\m^2$ such that $(x+t)+\p$ is transcendental over $R/(\p \cap R)$ for every $\p\in C$.  Define $S=R[x+t]_{\m\cap R[x+t]}$; we claim that $S$ is our desired $Z_d$-subring with distinguished set $\cal{Q}_S$, where we will define $\cal{Q}_S$ later.  

We first show that $\p\cap S=(\p\cap R)S$ for any $\p\in C$. Elements of $\p\cap S$ are of the form $uf$, where $u\in S^{\times}$ is a unit and $f\in R[x+t]$. Since $u$ is a unit, we have that $f\in\p$. Treating $f$ as a polynomial in $x+t$ over $R$, each of its coefficients must be in $\p \cap R$ since $x+t+\p$ is transcendental over $R\slash\p\cap R$. Thus, $f\in(\p\cap R)R[x+t]$, and so $uf\in(\p\cap R)S$. This gives us that $\p\cap S\subseteq(\p\cap R)S$, and the opposite containment is clear, so we have equality. 

Clearly, $|S|=\operatorname{sup}(\aleph_0,|R|)$ as we are simply adjoining a transcendental element and localizing, and so $S$ satisfies condition (i) of $Z_d$-subrings and condition (ii) of $A_+$-extensions.  Now let $\fq\in\ass T$; then $\fq\cap S=(\fq\cap R)S=(0)$, and so condition (ii) of $Z_d$-subrings is satisfied.  By the same argument, $S \cap P = (0)$.

Now, let $\p\in\ass (T/rT)$ for some regular $r\in T$.  First suppose $\p\cap R=(0)$. Then, in $S_{\p\cap S}=R[x+t]_{\p\cap R[x+t]}$, all nonzero elements of $R$ are units, and so $S_{\p\cap S}$ is isomorphic to $k[X]$ with additional elements inverted, where $k$ is a field and $X$ is an indeterminate.  It follows that $\dim S_{\p\cap S}\leq 1$ and we have $\height(\p\cap S)\leq 1$. Now suppose $\p\cap R=aR$ for some $a\in R$.  We know $\p\in\ass(T/rT)$ if and only if $\p T_{\p}\in\ass(T_{\p}/rT_{\p})$ (Theorem 6.2 of~\cite{matsumura86}), and so $\p T_{\p}$ consists of zerodivisors of $T_{\p}/rT_{\p}$. Then $T_{\p}/rT_{\p}$ consists only of zerodivisors and units and hence has depth zero, and since $a\in T$ is regular, $T_{\p}/aT_{\p}$ must also have depth zero.  Then $\p T_{\p}$ consists only of zerodivisors of $T_{\p}/aT_{\p}$, and so $\p T_{\p}\in\ass(T_{\p}/aT_{\p})$, and we have that $\p\in\ass(T/aT)$.  Therefore, $\p\in C$, so we have that $\p\cap S=(\p\cap R)S=aS$. Since $R$ is a UFD and $x + t$ is transcendental over $R$, $R[x+t]$ is a UFD as well, and furthermore, any localization of a UFD is a UFD, so $S=R[x+t]_{\m\cap R[x+t]}$ is a UFD.  Therefore, since $aS$ is a principal ideal of $S$, $\height(\p\cap S)\leq 1$ and $S$ satisfies condition (iii) of $Z_d$-subrings.

Note that, since $x + t$ is transcendental over $R$, prime elements of $R$ are prime in $S$.  Finally, we define the distinguished set $\cal{Q}_S$ for $S$.  We first show that there is a one-to-one correspondence between height one prime ideals of $R$ and height one prime ideals $pS$ of $S$ such that $pS \cap R \neq (0)$ given by $pR \mapsto pS$.
If $pS$ is a height one prime ideal of $S$ such that $pS \cap R \neq (0)$, then there is an element $s \in S$ such that $ps \in R$.  Factoring $ps$ into primes in $R$, we get $ps = q_1\cdots q_n$ where $q_i$ are prime elements in $R$.  Since prime elements in $R$ are prime in $S$, we have that $q_1\cdots q_n$ is a prime factorization of $ps$ in $S$.  It follows that $p = q_iu$ for some $i$ where $u$ is a unit in $S$.  Hence, $pS = (q_iu)S = q_iS$, so for every height one prime ideal $pS$ of $S$ with $pS \cap R \neq (0)$, there is a prime element $q$ of $R$ such that $pS = qS$, and it follows that our map is onto.  Now suppose that $p$ and $q$ are prime elements of $R$ such that $pS = qS$.  Then, letting $P' \in \ass(T/pT)$, we have that $pS \cap R \subseteq P' \cap R = pR$.  Clearly, $pR \subseteq pS \cap R.$  This gives us that $pR = pS \cap R = qS \cap R = qR$, showing that our map is one-to-one.  Let $\fq_p\in\cal{Q}_R$, where $p$ is a prime element of $R$ satisfying $\fq_p\cap R=pR$. Then $\fq_p\cap S=(\fq_p\cap R)S=pS$.  Hence, the map from $\cal{Q}_R$ to the set of height one prime ideals $pS$ of $S$ satisfying $pS \cap R \neq (0)$ given by $\fq_p \mapsto \fq_p \cap S = pS$ is a bijection.


Now let $p$ be a nonzero prime element of $S$ such that $pS \cap R = (0)$.
We claim that there is $\fq_p\in\spec_d T$ such that $\fq_p\cap S=pS$, and we
show this by induction on $d$.  If $d = 1$, let $\fq_p$ be a minimal prime ideal of $pT$.
In the case $d=2$, let $\p \in \spec T$ be a minimal prime ideal of $pT$, and consider the embedding
$S/\p \cap S\hookrightarrow T/\p$.
Since $T/\p $ is a Noetherian integral domain, each nonzero element of $S/\p \cap S$ is contained in
finitely many height one prime ideals of $T/\p $,
and since, by Lemma~\ref{modcard}, $|S/\p \cap S|<|T|=|T/\p  |$,
there must be some $\overline{\fq}\in\spec_1(T/\p )$ whose intersection with $S/\p \cap S$ is zero.  Lifting to $T$ and using the fact that $T$ is equidimensional and catenary, we have a height two prime ideal $\fq$ of $T$ containing $\p$
such that $\fq\cap S\subseteq\p\cap S$.
By condition (iii) of $Z_d$-subrings, $\p\cap S=pS$
and so $\fq\cap S=pS$.
Now let $2<d<\dim T$, and suppose that $\fq'\in\spec_{d-1} T$
satisfies $\fq'\cap S=pS$.
Applying the above argument to the injection $S/\fq'\cap S=S/pS\hookrightarrow T/\fq'$
we can find some $\fq_p\in\spec_d T$ containing $p$ and
satisfying $\fq_p\cap S\subseteq \fq' \cap S$,
and so $\fq_p\cap S=pS$ and the induction is complete.  Now for every height one prime ideal $pS$ of $S$ such that $pS \cap R = (0)$, choose one $\fq \in \spec_d T$ such that $\fq \cap S = pS$.
We define $\cal Q_S$ to be the union of the set of these prime ideals and $\cal Q_R$.
Then $S$ is our desired $A_+$-extension of $R$.
\end{proof}

\begin{lemma}\label{h4}
Under Assumption~\ref{assume}, let $\fa$ be a finitely generated ideal of $R$, and let $c\in\fa T\cap R$.  Let $P$  be a nonmaximal prime ideal of $T$ such that $P \cap R = (0)$.  Then there exists an $A_+$-extension $S$ of $R$ such that $c\in\fa S$ and $P \cap S = (0)$.
\end{lemma}
\begin{proof}
Note that this lemma is similar to Lemma 4 in~\cite{heitmann93} with the change that, instead of $R$ and $S$ being merely $N$-subrings, $R$ and $S$ are $Z_d$-subrings of $T$ with distinguished sets $\cal{Q}_R$ and $\cal{Q}_S$ satisfying $\cal Q_R\subseteq\cal Q_S$.  Therefore, most of the proof follows the proof of Lemma 4 in~\cite{heitmann93}, and we need only show that condition (iv) of Definition~\ref{zsub} holds for the $S$ we construct, that we can choose $\cal Q_S$ so that $\cal Q_R\subseteq\cal Q_S$, and that $P \cap S = (0)$.  We will proceed by inducting on the number of generators of $\fa$.  Let $\fa=(a_1, \ldots ,a_n)R$.

It is shown in the proof of Lemma 4 of~\cite{heitmann93} that, without loss of generality, we may assume that $\fa$ is not contained in a height one prime ideal of $R$, and so we will assume this for the rest of our proof.  If $n = 1$ this implies that $\fa = R$, so $S=R$ is the desired $A_+$-extension of $R$.

As in the proof of Lemma~\ref{cosets}, we define $C=\{\p\in\spec T\mid\p\in\ass(T/rT)\text{ with }0\neq r\in R\}\cup\ass T\cup\cal{Q}_R \cap \{P\}$.

Now, we consider the case $n=2$. 
Then $c=a_1t_1+a_2t_2$ for some $t_1,t_2\in T$, and
so $c=(t_1+a_2t)a_1+(t_2-a_1t)a_2$ for any $t\in T$. 
Following the proof of Lemma 4 in~\cite{heitmann93},
let $x_1=t_1+a_2t$ and $x_2=t_2-a_1t$ for some $t$ to be carefully chosen later.
We claim that if $\p\in C$, then at most one of $a_1$ or $a_2$ is contained in $\p$.
This is clear if $\p\cap R=(0)$, and if $\height(\p\cap R)=1$,
then $\p\cap R=pR$ for some prime element $p\in R$.
If $a_1$ and $a_2$ are both in $\p$ then $\fa \subset \p\cap R=pR$,
contradicting that $\fa$ is not contained in a height one prime ideal of $R$.  So we can assume that no $\p\in C$ contains both $a_1$ and $a_2$.
Define $\mathring{R}=R[a_2^{-1},x_1]\cap R[a_1^{-1},x_2]$ and $S=\mathring{R}_{\m\cap\mathring{R}}$.

If $\p\in C$, then $|R/\p\cap R|\leq|R|$ and so its algebraic closure in $T/\p$ has cardinality at most $|R|$ if $R$ is infinite and will be countable if $R$ is finite.  Now let $\p \in C$ and suppose $a_2 \notin \p$.  Then each choice of $t$ modulo $\p$ gives a different $x_1$ modulo $\p$, and so for all but $\max(\aleph_0,|R|)$ choices of $t$ modulo $\p$, the element $x_1 + \p$ of $T/\p$ will be transcendental over $R/\p\cap R$.
For each $\p\in C$ with $a_2\not\in\p$,
let $D^1_{\p}\subset T$ be a full set of coset representatives for those choices of $t$
such that $x_1+\p$ is algebraic over $R/\p\cap R$.
Similarly, for each $\p\in C$ that does not contain $a_1$
let $D^2_{\p}$ be a full set of coset representatives
for the elements $t$ that make $x_2+\p$ algebraic over $R/\p\cap R$,
and define $D=\bigcup\{D^1_{\p}\cup D^2_{\p}\mid\p\in C\}$.
Then $|C\times D|\leq \max(\aleph_0,|R|)$ and so using Lemma~\ref{countavoidance}
if $R$ is countable and Lemma~\ref{avoidance} otherwise, we can choose $t$ so that, for $\p \in C$, we have 
$x_1+\p$ is transcendental over $R/\p \cap R$ if $a_2 \not\in \p$ and $x_2+\p$ is transcendental over $R/\p\cap R$ if $a_1 \not\in \p$.  It is shown in Lemma 4 of~\cite{heitmann93} that for such a choice of $t$, $S$ is an N-subring and prime elements in $R$ are prime in $S$.

Suppose $f \in P \cap \mathring{R}$.  Then $f \in R[a_1^{-1},x_2]$, and so there is a positive integer $r$ such that $a_1^rf \in R[x_2]$.  By the way $t$ was chosen, $x_2 + P$ is transcendental over $R/R \cap P$ and so the coefficients of $a_1^rf$ are in $R \cap P = (0)$.  As $a_1$ is not a zerodivisor, we have that the coefficients of $f$ are all $0$.  It follows that $P \cap \mathring{R} = (0)$ and so we have $P \cap S = (0)$.


Now suppose $\fq \in \cal Q_R$.  If $a_1 \notin \fq$ then, by the way $t$ was chosen, $x_2 +\fq$ is transcendental over $R/\fq\cap R$.  Let $f \in \fq \cap \mathring{R}$.
Viewing $f$ as an element of $R[x_2,a_1^{-1}]$, for sufficiently large $k$ we have that
$(a_1)^k f \in R[x_2]$.  It follows that the coefficients of $(a_1)^k f$ are in $\fq \cap R = pR$ for some prime element $p$ of $R$.  Hence we have $f \in pR[x_2,a_1^{-1}]$.  Similarly, if $a_2 \not\in \fq$, then $f \in pR[x_1,a_2^{-1}]$.  If $a_2 \in \fq$ then $a_2 \in R \cap \fq = pR$, and so $p$ is a unit in $R[x_1,a_2^{-1}]$ and we have $pR[x_1,a_2^{-1}] = R[x_1,a_2^{-1}]$.  In any case, $f \in pR[x_2,a_1^{-1}] \cap p R[x_1,a_2^{-1}] = p\mathring{R}$, and it follows that $\fq \cap \mathring{R} = p\mathring{R}$ and so $\fq \cap S = pS$.  We will define $\cal Q_S$ to contain $\cal Q_R$; then each height one prime ideal of $S$ generated by a prime element of $R$ will have a corresponding element in $\cal Q_S$.
For every height one prime ideal of $S$ whose intersection with $R$ is the zero ideal, we use the same procedure as in Lemma~\ref{cosets}
to choose an appropriate corresponding $\fq_p \in \spec_d T$.
We define $\cal Q_S$ to be the union of the set of these prime ideals of $T$ with $\cal Q_R$.
Then $S$ is our desired $A_+$-extension of $R$ and the $n=2$ case is shown.

For $n>2$, we construct an $A_+$-extension $R''$ of $R$ with $R\subseteq R''\subset T$, $P \cap R'' = (0)$, and such that there exists $c^*\in R''$ and an $(n-1)$ generated ideal $\fb$ of $R''$ with $c^*\in\fb T$.  By induction, there exists an $A_+$-extension $S$ of $R''$ such that $c^* \in \fb S$ and $P \cap S = (0)$.  We then show that $c \in \fa S$, which will complete the proof. 

Let $\fa=(a_1, \ldots ,a_n)R$ and let $\fb=(a_1,\ldots,a_{n-1})R$.
We first assume that $\fb$ is not contained in a height one prime ideal of $R$.
Since $c\in\fa T\cap R$ we can write $c=\sum a_it_i$ where $t_i\in T$.
We define $\tilde{t}=t_n+\sum_{i=1}^{n-1}a_iu_i$ where $u_i\in T$ will be chosen later, 
and we define $c^*=c-\tilde{t}a_n$.  Note that $c^* \in \fb T$ as desired.  Now we work to define $R''$ so that it is an $A_+$-extension of $R$, $c^* \in R''$, and $P \cap R'' = (0)$.
Since we are assuming that $\fb$ is not contained in a height one prime ideal of $R$, we cannot have a $\p \in C$ that contains all of $a_1, a_2, \ldots , a_{n - 1}$.  Now, as in the $n = 2$ case, choose $u_1 \in T$ so that $t_n + a_1u_1$ is transcendental over $R/ \p \cap R$ for all $\p \in C$ satisfying $a_1 \not\in \p$.
We then repeat the process and choose $u_2$ so that
$t_n+a_1u_1+a_2u_2$ is transcendental over $R/\p\cap R$ whenever $\p\in C$ and $a_2\not\in\p$.
Now if $a_2\in\p$ and $a_1\not\in\p$, then $a_2u_2\in\p$
and $t_n+a_1u_1+a_2u_2+\p$ is transcendental over $R/\p\cap R$, and
so transcendental elements will remain transcendental as additional terms are added.
Continuing the process until all $u_i$ have been chosen,
 we obtain $\tilde{t}$ so that $\tilde{t} + \p$ is transcendental over $R/\p \cap R$ for all $\p \in C$.
We then define $R''=R[\tilde{t}]_{\m\cap R[\tilde{t}]}$, and we note that $c^* \in R''$.  As in the $n = 2$ case, we have that $P \cap R'' = (0)$.
Following the proof of Lemma~\ref{cosets},
$R''$ can be shown to be an $A_+$-extension of $R$.  By induction, there exists an $A_+$-extension $S$ of $R''$ such that $c^* \in \fb S$ and $P \cap S = (0)$.  Since $c = c^* + \tilde{t} a_n$, we have that $c \in \fa S$ as desired.

Now suppose that $\fb$ is contained in a height one prime ideal of $R$.
By factoring out common divisors 
we obtain an element $r\in R$ and an ideal $\fb^* = (z_1, \ldots, z_{n - 1})$ of $R$
such that $\fb=r\fb^*$, $a_i = rz_i$ for all $i = 1,2, \ldots, n-1$, and $\fb^*$ is not contained in a height one prime ideal of $R$.  Define $w = \sum_{i = 1}^{n - 1} t_iz_i$.  Then $c = rw + t_na_n$.  We now use the $n = 2$ case with $\fa = (r,a_n)$ to find an $A_+$-extension $R'''$ of $R$ such that $P \cap R''' = (0)$ and $c = v_1a_n + v_2r$ with $v_1$ and $v_2$ in $R'''$.  Note that $\fb^*R'''$ is not contained in a height one prime ideal of $R'''$, so we can use the previous case with $\fb = \fb ^*R'''$, $\fa = (z_1, z_2, \ldots ,z_{n - 1},a_n)R'''$ and $c = v_2$ to get our $A_+$-extension $R''$ of $R$, and our element $c^* \in \fb^* T$.  By induction, we get $S$.  We need only show that $c \in \fa S$.  In our case, $v_2 = w - \tilde{t}a_n$, and so $v_2 \in \fb^* S + a_nS$, and so we have $v_2r \in \fb S + ra_nS$.  It follows that $c = v_1a_n + v_2r \in \fa S$ as desired.
\end{proof}

Recall that, at the end of our construction, we want $\alpha(A,(0))= t$ and $\alpha(A,pA) = d - 1$ for all prime elements $p$ of $A$.  We will construct $A$ to be a UFD and to satisfy condition (iv) of $Z_d$-subrings, and so $\alpha(A,pA)\geq d-1$ for all prime elements $p$ of $A$.  For the case where $d - 1 = t$, we will construct $A$ so that $\alpha (A,(0)) \leq d - 1$.  In order to ensure this, we simply need $A$ to contain a nonzero element of every height $d$ prime ideal of $T$.  The following lemma allows us to adjoin nonzero elements of the height $d$ prime ideals of $T$ to $A$. Then, by Matsumura's inequality, we have $\alpha(A,(0)) \geq \alpha(A,pA)$ for all prime elements $p$ of $A$, and so we get that $d - 1 \geq \alpha(A,(0)) \geq \alpha(A,pA) \geq d - 1$ and we can conclude that $\alpha(A,(0))=\alpha(A,pA) = d - 1$ as desired.  For the case $d - 1 \neq t$, we construct $A$ so that there is a prime ideal $P$ of $T$ of height $t$ such that $P \cap A = (0)$, so that $\alpha(A,(0)) \geq t$.  In addition, if $\fq$ is a prime ideal of $T$ whose height is greater than $d$ and $\fq \not\subseteq P$, we use Lemma \ref{diffheights} to construct $A$ so that ht$(\fq \cap A) > 1$.  This ensures that $\alpha(A,(0)) \leq t$ and that, if $\p$ is a height one prime ideal of $A$, then $\alpha(A,\p) \leq d - 1$.

\begin{lemma}\label{dropgeneric}
Under Assumption~\ref{assume}, let $(T,\m)$ be such that, for all $\p\in\ass T$, $\height\p\leq d-1$, and suppose that $T$ satisfies the condition that if $z$ is a regular element of $T$ and $Q \in \ass(T/zT)$, then $\height Q \leq d$.  Let $P$  be a nonmaximal prime ideal of $T$ such that $P \cap R = (0)$, and let $\fq \in \spec_r T$ with $d \leq r \leq \dim T$ and such that $\fq \not\subseteq P$, and $\fq\cap R=(0)$.  Then there exists an $A_+$-extension $S$ of $R$ such that  $S \cap P = (0)$ and $\fq\cap S\neq (0)$.
\end{lemma}
\begin{proof}
Let $C$ be the same set detailed in the proof of Lemma~\ref{cosets}.  Our hypotheses give that $\fq \not\subseteq \p$ for all $\p \in C$.  Using Lemma~\ref{countavoidance} or Lemma~\ref{avoidance}, choose $x\in\fq$ such that $x+\p$ is transcendental over $R/\p\cap R$ for every $\p\in C$.  We let $S=R[x]_{\m\cap R[x]}$.  As in the proof of Lemma~\ref{cosets}, $S$ will be our desired $A_+$-extension of $R$.
\end{proof}

The following lemma allows us to control the formal fibers at height one prime ideals.  We want to make sure that there are no prime ideals of $T$ with height greater than $d$ in the formal fiber of a height one prime ideal of our final ring.  We do this in the following lemma by adjoining elements of prime ideals of $T$ of height greater than $d$.

\begin{lemma}\label{diffheights}
Under Assumption~\ref{assume}, let $(T,\m)$ be such that, for all $\p\in\ass T$, $\height\p\leq d-1$, and suppose that $T$ satisfies the condition that if $z$ is a regular element of $T$ and $Q \in \ass(T/zT)$, then $\height Q \leq d$.  Let $P$  be a nonmaximal prime ideal of $T$ such that $P \cap R = (0)$.  Let $\fq\in\spec_{r} T$ with $d + 1 \leq r \leq \dim T$, $\fq \not\subseteq P$, and $\height(\fq \cap R) \leq 1$.  Then there exists an $A_+$-extension $S$ of $R$ such that $P \cap S = (0)$ and $\height(\fq \cap S) > 1$. 
\end{lemma}
\begin{proof}
First suppose $\fq\cap R=(0)$. Then use Lemma~\ref{dropgeneric} to find an $A_+$-extension $S_0$ such that $\fq\cap S_0\neq(0)$ and $P \cap S_0 = (0)$.  Now, $\height(\fq\cap S_0)\geq 1$.  If $\height(\fq\cap S_0)>1$, then $S = S_0$ and we are done.  So consider the case where $\height(\fq\cap S_0)=1$.  Then, since $S_0$ is a UFD, $\fq\cap S_0=pS_0$ for some prime element $p\in S_0$.  Let $\fq_p$ be the element of $\cal{Q}_{S_0}$ corresponding to $pS_0$.  Now we let $C=\{\p\in\spec T\mid\p\in\ass(T/rT)\text{ with } 0\neq r\in S_0\}\cup\ass T\cup\cal{Q}_{S_0}\cup \{P\}$ and, as before, adjoin some element $x\in\fq$ such that $x+\p$ is transcendental over $S_0/\p\cap S_0$ for every $\p\in C$.  We let $S=S_0[x]_{\m\cap S_0[x]}$.  Then $S$ is an $A_+$-extension of $S_0$, which is an $A_+$-extension of $R$, and so $S$ is an $A_+$-extension of $R$.  Since $P \in C$, we have $P \cap S = (0)$.  Note that $pS \subseteq \fq \cap S$ and that $pS$ is a height one prime ideal of $S$.  Clearly, $x \in \fq \cap S$.  If $x \in pS$ then $x \in \fq_p$, contradicting that $x+\fq_p$ is transcendental over $S_0/\fq_p\cap S_0$.  It follows that $pS$ is strictly contained in $\fq \cap S$ and so ht$(\fq \cap S) > 1$.
\end{proof}

\begin{lemma} \label{doubleext}
Under Assumption~\ref{assume}, let $(T,\m)$ be such that $\depth T > 1$ and, for all $\p\in\ass T$, $\height\p\leq d-1$, and suppose that $T$ satisfies the condition that if $z$ is a regular element of $T$ and $Q \in \ass(T/zT)$, then $\height Q \leq d$. Let $P$  be a nonmaximal prime ideal of $T$ such that $P \cap R = (0)$, and let $t\in T$.  If $\fq\in\spec_dT$ with $\fq \not\subseteq P$, then there exists an $A_+$-extension S of R such that $t+\m^2\in\image(S\rightarrow T/\m^2)$, $P \cap S = (0)$, and $\fq\cap S\neq (0)$.  If $\fq\in\spec_{r}T$ with $d + 1 \leq r \leq \dim T$, and $\fq \not\subseteq P$, then there exists an $A_+$-extension S of R such that $t+\m^2\in\image(S\rightarrow T/\m^2)$, $P \cap S = (0)$, and ht$(\fq\cap S) > 1$.
\end{lemma}
\begin{proof}
If $\fq \in \spec_dT$, use Lemma~\ref{cosets} to find an $A_+$-extension $R'$ of $R$ such that $t+\m^2\in\image(R'\rightarrow T/\m^2)$ and $P \cap R' = (0)$.  If $\fq\cap R'\neq(0)$, then $S = R'$ and we are done. If $\fq\cap R'=(0)$, then we use Lemma~\ref{dropgeneric} to find an $A_+$-extension $S$ of $R'$ such that $\fq\cap S\neq(0)$ and $P \cap S = (0)$.

If $\fq\in\spec_{r}T$, where $d + 1 \leq r \leq \dim T$ then, as before, first use Lemma~\ref{cosets} to find an $A_+$-extension $R'$ of $R$ such that $t+\m^2\in\image(R'\rightarrow T/\m^2)$ and $P \cap R' = (0)$.  If ht$(\fq\cap R') > 1$, then $S = R'$ and we are done.  If ht$(\fq\cap R') \leq 1$ then use Lemma \ref{diffheights} to find an $A_+$-extension of $R'$ such that $P \cap S = (0)$ and ht$(\fq \cap S) > 1$.
\end{proof}

For the proof of our main theorem, we apply Lemmas~\ref{h4} and \ref{doubleext} infinitely often.  The following result, adapted from Lemma 6 of~\cite{heitmann93}, will allow us to do so.

\begin{lemma} 
\label{union}
Let $(T,\m)$ be a complete local ring and $R_0$ a $Z_d$-subring of $T$ with distinguished set $\cal{Q}_{R_0}$.  Let $P$ be a nonmaximal ideal of $T$ such that $P \cap R_0 = (0)$.  Let $\Omega$ be a well-ordered set with least element $0$ and assume either $\Omega$ is countable or, for all $\beta\in\Omega, |\{\gamma\in\Omega \mid\gamma<\beta\}|<|T/\m|$.  Suppose $\{R_{\beta}\mid \beta\in\Omega\}$ is an ascending collection of rings such that $R_{\alpha} \cap P = (0)$ for every $\alpha \in \Omega$ and such that, if $\beta$ is a limit ordinal, then $R_{\beta}=\bigcup_{\gamma<\beta}R_{\gamma}$,
with $\cal Q_{R_{\beta}}$ defined as $\bigcup_{\gamma<\beta}Q_{R_{\gamma}}$, 
while if $\beta=\gamma+1$ is a successor ordinal then $R_{\beta}$ is an $A_+$-extension of $R_{\gamma}$.

Then $S=\bigcup_{\beta \in \Omega}  R_{\beta}$ satisfies all the conditions to be a $Z_d$-subring of $T$ with distinguished set $\cal Q_{S} = \bigcup_{\beta \in \Omega} Q_{R_{\beta}} $except the cardinality condition. Instead, $|S|\leq\sup(\aleph_0, |R_0|,|\Omega|)$.  Furthermore, $P \cap S = (0)$, elements which are prime in some $R_{\beta}$ remain prime in $S$, and $\cal Q_{R_0} \subseteq \cal Q_S$.
\end{lemma}

\begin{proof}  First note that, since $R_{\alpha} \cap P = (0)$ for every $\alpha \in \Omega$, it is clear that $S \cap P = (0)$.  We now follow the proof of Lemma 6 in~\cite{heitmann93}, adding in additional steps where necessary.  Define $\Omega' =  \Omega \cup \{\delta\}$ and declare that $\delta > \alpha$ for all $\alpha \in \Omega$.  Now define $R_{\delta} = S$.
We will show that, for all $\alpha \in \Omega'$, $R_{\alpha}$ is a $Z_d$-subring of $T$ with some distinguished set $\cal{Q}_{R_{\alpha}}$ except for possibly condition (i) and that $|R_{\alpha}| \leq \sup(\aleph_0,|R_0|, |\{\beta \in \Omega \, | \, \beta < \alpha\}|)$.  Furthermore, we will show that, for $\beta < \alpha$, we have $\cal Q_{R_{\beta}} \subseteq \cal Q_{R_{\alpha}}$ and prime elements of $R_{\beta}$ remain prime in $R_{\alpha}$.  We proceed with transfinite induction, the base case being trivial.

Assume that $\alpha \in \Omega'$ and that the inductive hypotheses hold for every $\beta < \alpha$.  In the proof of Lemma 6 in~\cite{heitmann93}, it is shown that $R_{\alpha}$ satisfies the conditions for being a $Z_d$-subring of $T$ except for conditions (i) and (iv), that the cardinality condition given in the preceding paragraph holds, and that, if $\beta < \alpha$, every prime element of $R_{\beta}$ is prime in $R_{\alpha}$.  
If $\alpha = \gamma + 1$ is a successor ordinal, then $R_{\alpha}$ is an $A_+$-extension of $R_{\gamma}$, and so $R_{\alpha}$ satisfies condition (iv) of Definition~\ref{zsub} and $\cal Q_{R_{\gamma}} \subseteq \cal Q _{R_{\alpha}}$.  

If $\alpha$ is a limit ordinal, then $R_{\alpha} = \bigcup_{\gamma<\alpha}R_{\gamma}$, and we have $\cal Q_{R_{\alpha}} = \bigcup_{\gamma<\alpha} \cal Q_{R_{\gamma}}$.  By definition,
$\cal Q_{R_{\gamma}} \subseteq \cal Q _{R_{\alpha}}$ for all $\gamma < \alpha$.  We have left to show that $\cal{Q}_{R_{\alpha}}$ is a distinguished set for $R_{\alpha}$.  Let $\fq_p \in \cal{Q}_{R_{\alpha}}$.  Then $\fq_p \in \cal{Q}_{R_{\gamma}}$ for some $\gamma < \alpha$.  Therefore, $\fq_p \cap R_{\gamma} = pR_{\gamma}$ for some prime element $p$ of $R_{\gamma}$.  We will show that $\fq_p \cap R_{\alpha} = pR_{\alpha}$.  Clearly, $pR_{\alpha} \subseteq \fq_p \cap R_{\alpha}$.  If $x \in R_{\alpha} \cap \fq_p$, then $x \in R_{\beta}$ for some $\beta < \alpha$.  Define $\lambda = \max\{\beta,\gamma\}$.  Then $p,x \in R_{\lambda}$ and $\fq_p \in \cal{Q}_{R_{\lambda}}$.  It follows that $\fq_p \cap R_{\lambda} = pR_{\lambda}$, so $x \in \fq_p \cap R_{\lambda} = pR_{\lambda} \subseteq pR_{\alpha}$, and we have that $\fq_p \cap R_{\alpha} = pR_{\alpha}$ as desired.  It is not difficult to show that if $\fq_p, \fq_p' \in \cal{Q}_{R_{\alpha}}$ with $\fq_p \cap R_{\alpha} = \fq_p' \cap R_{\alpha}$ then $\fq_p = \fq_p'$, and that if $\p = pR_{\alpha}$ is a height one prime ideal of $R_{\alpha}$, then there is a $\fq_p \in \cal{Q}_{R_{\alpha}}$ such that $\fq_p \cap R_{\alpha} = \p$.  Hence $\cal{Q}_{R_{\alpha}}$ is a distinguished set for $R_{\alpha}$.

By induction, then, we have that $S$ satisfies all of the desired properties.
\end{proof}

\begin{lemma}\label{closeideals}
Under Assumption~\ref{assume}, let $(T,\m)$ be such that $\depth T > 1$ and, for all $\p\in\ass T$, $\height\p\leq d-1$, and suppose that $T$ satisfies the condition that if $z$ is a regular element of $T$ and $Q \in \ass(T/zT)$, then $\height Q \leq d$.  Let $P$  be a nonmaximal prime ideal of $T$ such that $P \cap R = (0)$ and let $t+\m^2\in T/\m^2$.  If $\fq\in\spec_d T$ with $\fq \not\subseteq P$, then there exists an $A_+$-extension $S$ of $R$ such that $t+\m^2\in\image(S\rightarrow T/\m^2)$, $\fq\cap S\neq(0)$, $P \cap S = (0)$, and for every finitely generated ideal $\fa$ of $S$, $\fa T\cap S=\fa$.  If $\fq\in\spec_{r} T$ with $d + 1 \leq r \leq \dim T$ and $\fq \not\subseteq P$, then there exists an $A_+$-extension $S$ of $R$ such that $t+\m^2\in\image(S\rightarrow T/\m^2)$, ht$(\fq\cap S) > 1$, $P \cap S = (0)$, and for every finitely generated ideal $\fa$ of $S$, $\fa T\cap S=\fa$.
\end{lemma}
\begin{proof}
If $\fq \in\spec_d T$, employ Lemma~\ref{doubleext} to obtain an $A_+$-extension $R_0$ of $R$ such that $t+\m^2\in\image(R_0\rightarrow T/\m^2)$, $P \cap R_0 = (0)$, and $\fq\cap R_0\neq(0)$. 
If $\fq \in\spec_{r} T$ with $d + 1 \leq r \leq \dim T$, employ Lemma~\ref{doubleext} to obtain an $A_+$-extension $R_0$ of $R$ such that $t+\m^2\in\image(R_0\rightarrow T/\m^2)$, $P \cap R_0 = (0)$, and ht$(\fq\cap R_0) > 1$.  Let $\Omega=\{(\fa,c)\mid\fa$ finitely generated ideal of $R_0 \text{ and }c\in \fa T\cap R_0\}$. Then $|\Omega|=|R_0|$ and so either $\Omega$ is countable or $|\Omega|<|T/\m|$. Well-order $\Omega$, letting 0 designate its initial  element, in such a way that $\Omega$ does not have a maximal element; then it clearly satisfies the hypothesis of Lemma~\ref{union}.  We will recursively define an increasing chain of rings with one ring for every element of $\Omega$.  We begin with $R_0$. If $\beta=\gamma+1$ is a successor ordinal and $\gamma=(\fa,c)$, then we choose $R_{\beta}$ to be an $A_+$-extension of $R_{\gamma}$ given by Lemma~\ref{h4} such that $c\in\fa R_{\beta}$ and $P \cap R_{\beta} = (0)$. If $\beta$ is a limit ordinal, define $R_{\beta}=\bigcup_{\gamma<\beta}R_{\gamma}$ and $\cal Q_{R_{\beta}} = \bigcup_{\gamma<\beta}\cal Q_{R_{\gamma}}$. Set $R_1=\bigcup R_{\beta}$.  By Lemma~\ref{union}, we see that $R_1$ is an $A_+$-extension of $R_0$ and $P \cap R_1 = (0)$.  Also if $\fa$ is any finitely generated ideal of $R_0$ and $c\in \fa T\cap R_0$, then $(\fa,c)=\gamma$ for some $\gamma \in \Omega$.  Then for some $\beta>\gamma$, $c\in\fa R_{\beta}\subseteq \fa R_1$.  Thus $\fa T\cap R_0\subseteq\fa R_1$.

We repeat the process to obtain an $A_+$-extension $R_2$ of $R_1$ such that $P \cap R_2 = (0)$ and $\fa T\cap R_1\subseteq\fa R_2$ for every finitely generated ideal $\fa$ of $R_1$.  Continue recursively to obtain an ascending chain $R_0\subseteq R_1\subseteq \cdots$ such that $P \cap R_n = (0)$ and $\fa T\cap R_n\subset\fa R_{n+1}$ for every finitely generated ideal $\fa$ of $R_n$.  Then by Lemma~\ref{union}, $S=\bigcup R_i$ with $\cal Q_{S} = \bigcup \cal Q_{R_{i}}$ is an $A_+$-extension of $R_0$, and so also of $R$, and $P \cap S = (0)$.  Further, if $\fa$ is a finitely generated ideal of $S$, then some $R_n$ contains a generating set $\{a_1, \ldots ,a_k\}$ for $\fa$.  If $c\in\fa T\cap S$, then $c\in R_m$ for some $m\geq n$, so $c\in(a_1,\ldots ,a_k)T\cap R_m   \subseteq (a_1, \ldots ,a_k)R_{m+1}\subseteq \fa$. Thus $\fa T\cap S=\fa$.  To see that ht$(\fq\cap R_0) > 1$ implies ht$(\fq\cap S) > 1$, let $(0) \subset pR_0 \subset \fq \cap R_0$ be a strictly increasing chain of prime ideals of $R_0$, and let $x \in \fq \cap R_0$ with $x \not\in pR_0$.  As prime elements in $R_0$ are prime in $S$, the prime factorization of $x$ in $R_0$ is the prime factorization of $x$ in $S$ and so $x \not\in pS$ and we have $(0) \subset pS \subset \fq \cap S$ is a strictly increasing chain of prime ideals of $S$.
\end{proof}


\section{The Main Theorem and Corollaries}

\begin{theorem}\label{bigthm}
  Let $(T,\m)$ be a complete local equidimensional ring such that $\dim T\geq 2$ and $\depth T>1$.  Suppose that no integer of $T$ is a zerodivisor in $T$ and $|T|=|T/\m|$.  Let $d$ and $t$ be integers such that $1 \leq d \leq \dim T - 1$, $0 \leq t \leq \dim T - 1$, and $d - 1 \leq t$.  Assume that, for every $\p \in \ass T$, ht $\p \leq d - 1$ and that if $z$ is a regular element of $T$ and $Q \in \ass (T/zT)$, then ht $Q \leq d$.  Then there exists a local UFD A such that $\widehat{A} = T$, $\alpha(A,(0)) = t$, and, if $\p \in \spec_1A$, then $\alpha(A,\p) = d - 1$.

\end{theorem}
\begin{proof}
First assume $t = d - 1$.  Let $\Omega_1=\spec_d T$, well-ordered so that each element of $\Omega_1$ has fewer than $|\Omega_1|$ predecessors.  
Let $\Omega_2=T/\m^2$, well-ordered so that each element of $\Omega_2$ has fewer than $|\Omega_2|$ predecessors.  Because $|\Omega_1|=|\Omega_2|$ and we are ordering both sets so that each element of $\Omega_i$ has fewer than $|\Omega_i|$ predecessors, we can use $\mathscr{B}$ as the index set for both of the $\Omega_i$'s.  Let 

\[
\Omega=\{(\fq_a,t_a)\mid \fq_a\in\Omega_1,t_a\in\Omega_2,\text{ where }a\in\mathscr{B}\}
\]

\noindent well-ordered using $\mathscr{B}$ as the index set.  Then $\Omega$ is the diagonal of $\Omega_1\times\Omega_2$. Let 0 designate the first element of $\Omega$.  Now, $\Omega$ satisfies the hypothesis of Lemma~\ref{union}, and we now recursively define a family of rings $\{R_{\beta}\mid\beta\in\Omega\}$ which also satisfies the hypotheses of Lemma~\ref{union}.  As in the proof of Theorem 8 in~\cite{heitmann93}, let $R_0$ be the appropriate localization of the prime subring of $T$: either $\mathbb{Q}$, $\mathbb{Z}_{p}$, or $\mathbb{Z}_{(p)}$ where $p$ is a prime integer.  It is not difficult to verify that $R_0$ is an $N$-subring.  Now, $R_0$ is either dimension $0$ or $1$.  If it is dimension $0$, define $\cal Q_{R_0}$ to be the empty set, and if it is dimension $1$, define $\cal Q_{R_0}$ to be any height $d$ prime ideal of $T$ that contains a minimal associated prime ideal of $pT$.   Then $R_0$ is a $Z_d$-subring of $T$ with distinguished set $\cal Q_{R_0}$.  We will use the shorthand $\beta=(\fq_{\beta},t_{\beta})$ for an element of $\Omega$.  Then, whenever $\beta=\omega+1$ is a successor ordinal, we let $R_{\beta}$ be an $A_+$-extension of $R_{\omega}$ chosen in accordance with Lemma~\ref{closeideals} so that $\fq_{\omega}\cap R_{\beta}\neq(0)$, $t_{\omega}\in\image(R_{\beta}\rightarrow T/\m^2)$, and $\fa T \cap R_{\beta} = \fa$ for every finitely generated ideal $\fa$ of $R_{\beta}$.  Note that in this case (where $t = d - 1$) we do not need to use the $P$ given in Lemma \ref{closeideals}.  If $\beta$ is a limit ordinal, choose $R_{\beta}=\bigcup_{\gamma<\beta}R_{\gamma}$.  We claim $A=\bigcup_{\beta \in \Omega} R_{\beta}$ is the desired example.

By construction, $A\rightarrow T/\m^2$ is onto.  Now let $\fa = (a_1,\ldots ,a_n)$ be a finitely generated ideal of $A$, and let $x \in \fa T \cap A$. Then there is a $\beta \in \Omega$ such that $\beta$ is a successor ordinal and $x, a_1, \ldots , a_n \in R_{\beta}$, so $x \in (a_1, \ldots ,a_n)T \cap R_{\beta} = (a_1, \ldots ,a_n)R_{\beta} \subseteq \fa$.  It follows that $\fa T\cap A=\fa$ for all finitely generated ideals of $A$.  Then by Lemma~\ref{comp}, $\widehat{A}= T$ and $A$ is Noetherian. By Lemma~\ref{union}, except for the cardinality condition, $A$ satisfies all the conditions of  being a $Z_d$-subring of $T$ with some distinguished set, which implies that for any $\p\in\spec_1A$, $\alpha(A,\p)\geq d-1$. Now let $\fq\in\spec_dT$.  Then by our construction there is some $R_{\beta}$ such that $\fq\cap R_{\beta}\neq(0)$, and so $\fq\cap A\neq(0)$.  Thus $\alpha(A,(0))\leq d-1$.  Then, by Theorem 1 of~\cite{matsumura88}, $d-1\geq\alpha(A,(0))\geq\alpha(A,\p)\geq d-1.$  Thus, for every $\p\in\spec_1A$, $\alpha(A,\p)=\alpha(A,(0))=d-1$.

Now we consider the case $t \neq d - 1$.  In this case, we adjust the construction described above.  We first find a prime ideal $P$ of $T$ such that ht$P = t$ and $P \cap R_0 = (0)$.  If $R_0 = \mathbb{Q}$ or $R_0 = \mathbb{Z}_p$, then all nonzero elements of $R_0$ are units of $T$, and so we can choose $P$ to be any prime ideal of $T$ with height $t$.   If $R_0 = \mathbb{Z}_{(p)}$ and $t = 0$, then choose $P$ to be any minimal prime ideal of $T$.  Since $p$ is not a zerodivisor of $T$, we have $P \cap R_0 = (0)$.  If $R_0 = \mathbb{Z}_{(p)}$, and $t > 0$ then note that $p$ is contained in finitely many height one prime ideals of $T$, and so there is a height one prime ideal $P_1$ such that $P_1 \cap R_0 = (0)$.  Now the element $p + P_1$ in $T/P_1$ is contained in finitely many height one prime ideals of $T/P_1$, and so there is a height one prime ideal $\bar{P}_2$ of $T/P_1$ such that $p + P_1 \not\in \bar{P}_2$.  As $T$ is universally catenary and equidimensional, there is a height two prime ideal $P_2$ of $T$ such that $R_0 \cap P_2 = (0)$.  Continue in this way to find a height $t$ prime ideal $P$ of $T$ such that $P \cap R_0 = (0)$.

Now suppose $t = d = \dim T - 1$.  Then use the same construction described for the $t = d - 1$ case with the following adjustments.  Let $\Omega_1 = \spec_d T - \{P\}$ where $P$ is the height $t$ prime ideal of $T$ chosen in the above paragraph.  Then, in the construction, whenever $\beta=\omega+1$ is a successor ordinal, let $R_{\beta}$ be an $A_+$-extension of $R_{\omega}$ chosen in accordance with Lemma~\ref{closeideals} so that $\fq_{\omega}\cap R_{\beta}\neq(0)$, $P \cap R_{\beta} = (0)$, $t_{\omega}\in\image(R_{\beta}\rightarrow T/\m^2)$, and $\fa T \cap R_{\beta} = \fa R_{\beta}$ for every finitely generated ideal $\fa$ of $R_{\omega}$.  Then, using this adjusted construction, we get a UFD $A$ such that $\widehat{A} = T$, $P \cap A = (0)$, and for every $\p \in \spec_1 A$, there is a height $d = \dim T - 1$ prime ideal of $T$ whose intersection with $A$ is $\p$.  It follows that $\alpha(A,(0)) = t$ and $\alpha(A,\p) = d - 1$ for all $\p \in \spec_1 A$.

Finally, consider all other cases.  That is, suppose either $t > d$ or $t = d$ with $t < \dim T - 1$.  Then use the same construction described for the $t = d - 1$ case with the following adjustments.  Let
$$\Omega_1 = \{ \fq \in \spec_r T \, | \, d + 1 \leq r \leq \dim T \mbox{ and } \fq \not\subseteq P  \}$$
where $P$ is the height $t$ prime ideal of $T$ chosen previously.  Then, in the construction, whenever $\beta=\omega+1$ is a successor ordinal, let $R_{\beta}$ be an $A_+$-extension of $R_{\omega}$ chosen in accordance with Lemma~\ref{closeideals} so that ht$(\fq_{\omega}\cap R_{\beta}) > 1$, $P \cap R_{\beta} = (0)$, $t_{\omega}\in\image(R_{\beta}\rightarrow T/\m^2)$, and $\fa T \cap R_{\beta} = \fa R_{\beta}$ for every finitely generated ideal $\fa$ of $R_{\beta}$.  Then, using this adjusted construction, we get a UFD $A$ such that $\widehat{A} = T$, $P \cap A = (0)$, and for every $\p \in \spec_1 A$, there is a height $d$ prime ideal of $T$ whose intersection with $A$ is $\p$.  Moreover, if $\fq \in \spec_r T$ with $d + 1 \leq r \leq \dim T$ and $\fq \not\subseteq P$, then ht$(\fq \cap A) > 1$.  It follows that $\alpha(A,(0)) = t$ and $\alpha(A,\p) = d - 1$ for all $\p \in \spec_1 A$.
\end{proof}



\begin{corollary}\label{equal}
Let $(T,\m)$ be a complete local equidimensional ring such that $\dim T\geq 2$, no integer is a zerodivisor in $T$, $\depth T>1$, and $|T|=|T/\m|$. Let $d$ be an integer such that $1\leq d\leq\dim T-1$ and assume that, for every $\p\in\ass T$, $\height\p\leq d-1$ and that if $z$ is a regular element of $T$ and $Q \in \ass(T/zT)$, then $\height Q \leq d$. Then there exists a local unique factorization domain $A$ such that $\widehat{A}= T$ and for every $pA\in\spec_1A$, $\alpha(A,pA)=\alpha(A,(0))=d-1$.
\end{corollary}

Note that Corollary \ref{equal} answers the Question stated in the introduction posed by Heinzer, Rotthaus, and Sally in the nonexcellent case.  In particular, it shows that there are nonexcellent local integral domains where the dimension of the generic formal fiber ring is positive and where the set

$$\Delta = \{\p\in\spec A\mid\height\p=1\text{ and }\alpha(A,\p) = \alpha(A)\}$$

\noindent is equal to the set of all of the height one prime ideals of the integral domain.

\begin{example}
Let $T=\mathbb{C}[[x_1, \ldots ,x_n]]$.  Then, for any $1\leq d\leq n-1$, $0 \leq t \leq n - 1$ and $d - 1 \leq t$, we can find a local unique factorization domain $A$ such that $\widehat{A}= T$, $\alpha(A,(0))= t$, and $\alpha(A,pA)=d-1$ for every $pA\in\spec_1 A$.  Note that this will also hold true if we replace $\mathbb{C}$ with any uncountable field.
\end{example}

\begin{example}
Let $T=\mathbb{C}[[x,y,z,w]]/(xy-zw)$.  Then there exists a local unique factorization domain $A$ such that $\widehat{A}= T$ and $\alpha(A,pA)=\alpha(A,(0))=2$ for every $pA\in\spec_1 A$.
\end{example}

\begin{example}
Let $T=\mathbb{C}[[v,w,x,y,z]]/(vx-w^2, xz - y^2)$.  Then there exists a local unique factorization domain $A$ such that $\widehat{A}= T$, $\alpha(A,(0))=2$, and $\alpha(A,pA) = 0$ for every $pA\in\spec_1 A$.
\end{example}

\bibliography{references}

\begin{bibdiv}
\begin{biblist}

\bib{boocher10}{article}{
      author={Boocher, Adam},
      author={Daub, Michael},
      author={Loepp, S.},
       title={Dimensions of formal fibers of height one prime ideals},
        date={2010},
     journal={Communications in Algebra},
      volume={38},
      number={1},
       pages={233\ndash 253},
}

\bib{charters04}{article}{
      author={Charters, P.},
      author={Loepp, S.},
       title={Semilocal generic formal fibers},
        date={2004},
     journal={Journal of Algebra},
      volume={278},
      number={1},
       pages={370\ndash 382},
}

\bib{heitmann93}{article}{
      author={Heitmann, Raymond~C.},
       title={Characterizations of completions of unique factorization
  domains},
        date={1993},
     journal={Transactions of the AMS},
      volume={337},
      number={1},
       pages={379\ndash 387},
}

\bib{heitmann94}{article}{
      author={Heitmann, Raymond~C.},
       title={Completions of local rings with an isolated singularity},
        date={1994},
     journal={Journal of Algebra},
      volume={163},
      number={2},
       pages={538\ndash 567},
}

\bib{matsumura86}{book}{
      author={Matsumura, Hideyuki},
       title={Commutative ring theory},
   publisher={Cambridge University Press},
        date={1986},
}

\bib{matsumura88}{incollection}{
      author={Matsumura, Hideyuki},
       title={On the dimension of formal fibres of a local ring},
        date={1988},
   booktitle={Algebraic geometry and commutative algebra in honor of masayoshi
  nagata},
      editor={Hiroaki~Hijikata, Heisuke~Hironaka},
      editor={Maruyama, Masaki},
      volume={1},
   publisher={Academic Press},
}

\end{biblist}
\end{bibdiv}

\end{document}